\documentclass[12pt]{article}
\usepackage{graphicx,float,dsfont,graphics,citesort}
\usepackage[margin=1in]{geometry}
\usepackage{xcolor,tikz}
\usepackage{tikz}\usetikzlibrary{shapes,arrows}
\usepackage{amssymb,amsmath,amsthm,bm,setspace,mathrsfs}
\usepackage{url,color,units}
\usepackage[bf]{caption}
\usepackage{subcaption}
\usepackage[T1]{fontenc}

\newcommand{\Z}{\mathbb{Z}}

\newcommand{\R}{\mathbb{R}}
\newcommand{\C}{\mathbb{C}}

\newcommand{\cE}{\mathscr{E}}

\newcommand{\cG}{\mathcal{G}}
\newcommand{\cH}{\mathcal{H}}
\newcommand{\cI}{\mathcal{I}}

\newcommand{\cP}{\mathcal{P}}
\newcommand{\cR}{\mathcal{R}}
\newcommand{\cS}{\mathcal{S}}

\newcommand{\Log}{{\ensuremath{\rm Log}\,}}
\renewcommand{\P}{\mathrm{P}}
\newcommand{\E}{\mathrm{E}}
\newcommand{\U}{\mathrm{U}}
\newcommand{\1}{\mathds{1}}
\renewcommand{\d}{{\rm d}}
\newcommand{\e}{{\rm e}}
\DeclareMathOperator{\Var}{Var}

\DeclareMathOperator{\pix}{\rm pix}

\newcommand{\DimM}{\text{\rm Dim}_{_{\rm M}}}

\renewcommand{\ge}{\geqslant}
\renewcommand{\le}{\leqslant}
\renewcommand{\Re}{{\rm Re}}

\author{Davar Khoshnevisan\\University of Utah
	\and
		Yimin Xiao\\Michigan State University}
\title{{On the Macroscopic Fractal Geometry\\of Some Random Sets}\thanks{%
	Research supported in part by the National Science Foundation grant DMS-1307470}}
\date{Version: May 3, 2016}
\newtheorem{stat}{Statement}[section]
\newtheorem{proposition}[stat]{Proposition}
\newtheorem{corollary}[stat]{Corollary}
\newtheorem{theorem}[stat]{Theorem}
\newtheorem{lemma}[stat]{Lemma}
\theoremstyle{definition}
\newtheorem{definition}[stat]{Definition}

\newtheorem{remark}[stat]{Remark}
\newtheorem*{OP}{Open Problem}
\newtheorem{example}[stat]{Example}

\begin{document}
\maketitle
\begin{abstract}
	This paper is concerned mainly
	with the macroscopic fractal behavior of various random sets
	that arise in modern and classical probability theory. Among other things,
	it is shown here that the macroscopic behavior of Boolean coverage processes 
	is analogous to the microscopic structure of the Mandelbrot fractal percolation.
	Other, more technically challenging, results of this paper
	include:
	\begin{itemize}
		\item[(i)] The computation of the macroscopic dimension of the graph of
			a large family of L\'evy processes; and 
		\item[(ii)] The determination of the macroscopic monofractality
			of the extreme values of symmetric stable processes.
	\end{itemize}
	As a consequence of (i), it will be shown that the macroscopic
	fractal dimension of the graph of Brownian motion differs from its microscopic
	fractal dimension. Thus, there can be no scaling argument that
	allows one to deduce the macroscopic geometry from the microscopic. 
	Item (ii) extends the recent work of 
	Khoshnevisan, Kim, and Xiao  \cite{KKX} on the 
	extreme values of Brownian motion, using a different method.\\


\end{abstract}

\section{Introduction}
It has been known for some time that the curve of a L\'evy process
in $\R^d$ is typically an interesting ``random fractal.'' For example,
if $B=\{B_t\}_{t\ge0}$ is a standard Brownian motion on $\R^d$,
then the image and graph of $B$ have Hausdorff dimension
$d\wedge 2$ and $\max(d\wedge2\,,\nicefrac32)$ respectively. If in addition
$d=1$, then the level sets of $B$ also have non-trivial Hausdorff dimension
$\nicefrac12$. See the survey papers of Taylor \cite{Tay86} and Xiao \cite{X04} 
for historic accounts on these results and further developments. 

The beginning student is often presented with some of these
``random-fractal facts'' via simulation. The well-versed reader will
see in Figure \ref{fig:BM} a typical example. As a consequence
of such a simulation, one is led to believe that one can deduce from 
a simulation, such as that in Figure
\ref{fig:BM}, the fractal nature of the graph
$\cup_{0\le t\le 1}\{ (t\,,B_t)\}$ of Brownian motion up to time $1$.

\begin{figure}[h!]\centering
	\includegraphics[width=0.7\textwidth]{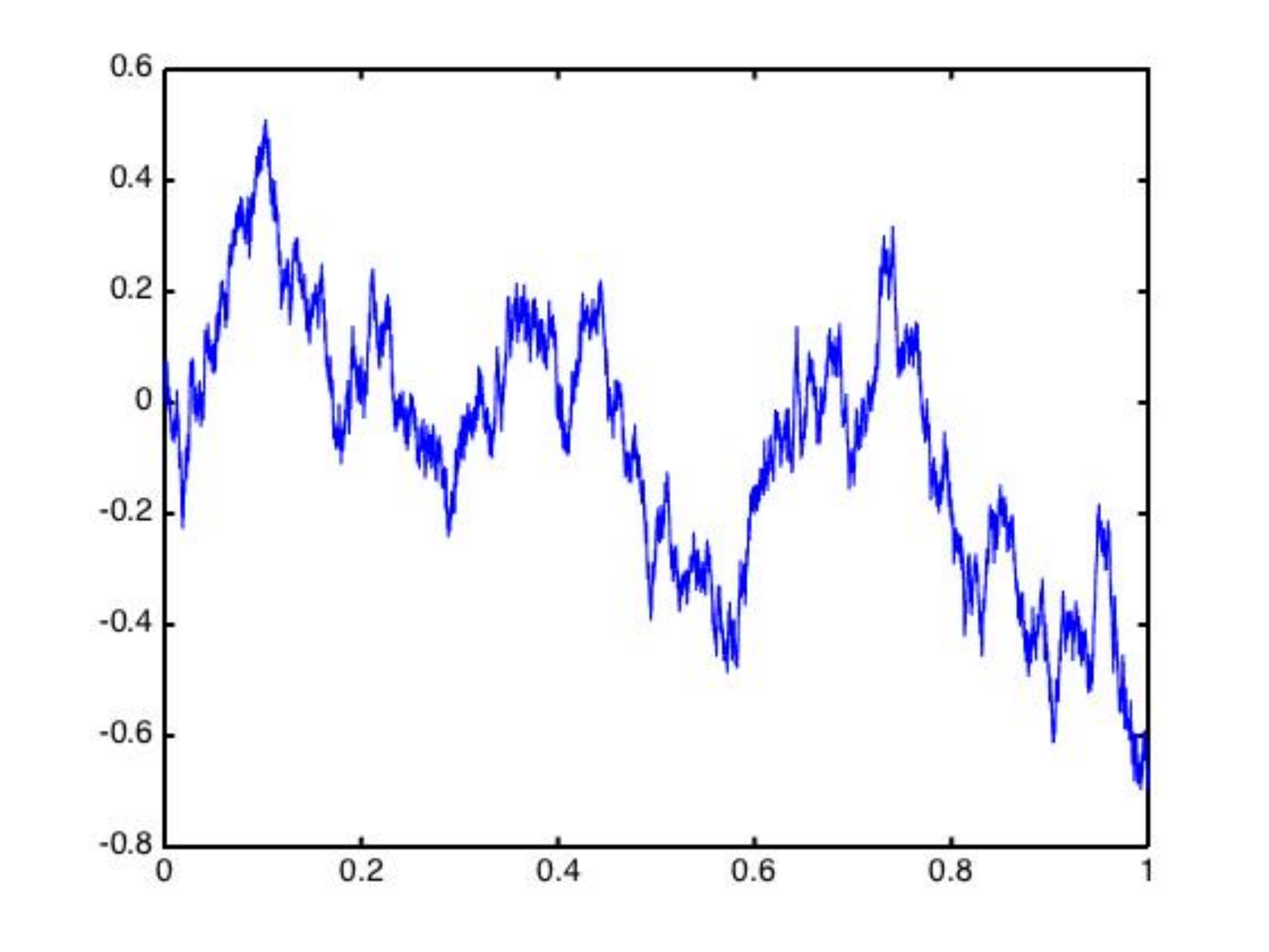}
\caption{The graph of planar Brownian motion}
\label{fig:BM}
\end{figure}

Figure \ref{fig:BM}, and other such simulations, are produced by running a
random walk for a long time and then rescaling, using a central-limit scaling. 
The process is usually explained by appealing to Donsker's invariance principle.
Unfortunately,
the actual statement of Donsker's invariance principle is not sufficiently strong 
to ensure that we can ``see'' the various fractal properties of Brownian motion
in simulations. Though Barlow and Taylor \cite{BarlowTaylor1,BarlowTaylor} 
have introduced a theory of large-scale random fractals which, 
among other things, provides a more rigorous justification. 

One of the goals of this paper is to test the extent to which 
one can experimentally deduce large-scale geometric facts about Brownian motion---%
and sometimes more general L\'evy processes---from simulation analysis.
This is achieved by presenting several examples in which one is able to 
compute the macroscopic fractal dimension of a macroscopic random fractal.
One of the surprising lessons of this exercise is that our intuition is, at times, faulty.
Yet, our instincts are correct at other times. 

Here is an example in which our intuition is spot on:
It is known that 
the level sets of Brownian motion have dimension $\nicefrac12$, both macroscopically
and microscopically. This statement has the pleasant consequence that
we can ``see'' the fractal structure of the level sets of Brownian motion from
Figure \ref{fig:BM}. As we shall soon see, however, the same cannot be
said of the graph of Brownian motion: The microscopic and macroscopic 
fractal dimensions of the graph of Brownian motion do not agree!

In order to keep the technical level of the paper as low as possible,
our choice of ``fractal dimension'' is the macroscopic Minkowski dimension,
which we will present in the following section. There are more sophisticated
notions which, we however, will not present here; see Barlow and Taylor
\cite{BarlowTaylor1,BarlowTaylor} for examples of these more sophisticated notions
of macroscopic fractal dimension.

Throughout, $\Log$ denotes the base-2 logarithm.
For all $x\in\R^d$, we set
$|x|:=\max_{1\le j\le d}|x_j|$
and $\|x\|:=(x_1^2+\cdots+x_d^2)^{1/2}$.
Whenever we write ``$f(x) \lesssim g(x)$ for all $x\in X$,'' and/or ``$g(x)\gtrsim f(x)$
for all $x\in X$,''
we mean that there exists a finite constant $c$ such that
$f(x) \le c g(x)$ uniformly for all $x\in X$. If
$f(x)\lesssim g(x)$ and $g(x)\lesssim f(x)$ for all $x\in X$, then
we write ``$f(x)\asymp g(x)$ for all $x\in X$.''

\section{Minkowski Dimension}
The macroscopic Minkowski dimension is an easy-to-compute
``fractal dimension number'' that describes the large-scale fractal geometry
of a set. In order to recall the  Minkowski dimension we first need
to introduce some notation.

For all $x\in\R^d$ and $r>0$ define
\[
	B(x;r):= [x_1-r\,,x_1+r)\times\cdots\times[x_d-r\,,x_d+r),
\]
and
\begin{equation}\label{Q}
	Q(x):= [x_1\,,x_1+1)\times\cdots\times[x_d\,,x_d+1).
\end{equation}
Of course, $Q(x)=B(y;\frac12)$ where
$y_i:=x_i+\frac12$. But it is convenient for $Q(x)$
to have its own notation.

One can introduce a \emph{pixelization map} which maps a set $F\subseteq\R^d$
to a set $\pix(F)\subseteq\Z^d$ as follows:
\[
	\pix(F) := \left\{ x\in \Z^d:\, F\cap Q(x)\neq\varnothing\right\},
\]
for all $F\subseteq\R^d$. 
It  is clear that $F=\pix(F)$ if and only if
$F$ is a subset of the integer lattice $\Z^d$. For example,
it should be clear that $\pix(\R^d)=\Z^d$.
Figure \ref{fig:pix} below shows how the pixelization map 
works in a different simple case.

\begin{figure}[h!]\centering
	\begin{tikzpicture}[scale=.8]
		\draw[color=red!20,step=1,ultra thin] (0,0) grid (16,8);
		\draw[thick,fill=red,opacity=0.2] (3,4) ellipse (1.9 and 3.2);
		\draw[thick,brown!90] (3,4) ellipse (1.9 and 3.2);
		\draw[->,ultra thick,blue] (5.1,4) to [bend right=30] node [above,black]{$\pix$} (10.8,4);
		\foreach \i in {11,12,13,14} {
			\foreach \j in {1,2,3,4,5,6}{
			\shade [ball color=red!90!blue!90] (\i,\j) circle [radius=1mm];
		}}
		\shade [ball color=red!90!blue!90] (12,7) circle [radius=1mm];
		\shade [ball color=red!90!blue!90] (13,7) circle [radius=1mm];
		\foreach \i in {1,2,3,4} {
			\foreach \j in {1,2,3,4,5,6}{
			\shade [ball color=red!50] (\i,\j) circle [radius=0.5mm];
		}}
		\shade [ball color=red!20] (2,7) circle [radius=0.5mm];
		\shade [ball color=red!20] (3,7) circle [radius=0.5mm];
		\draw[very thick] (0,0) to (16,0) to (16,8) to (0,8) to (0,0);
	\end{tikzpicture}
	\caption{The effect of the pixelization map on the ellipse on the left
	is shown on the right-hand side}
	\label{fig:pix}
\end{figure}
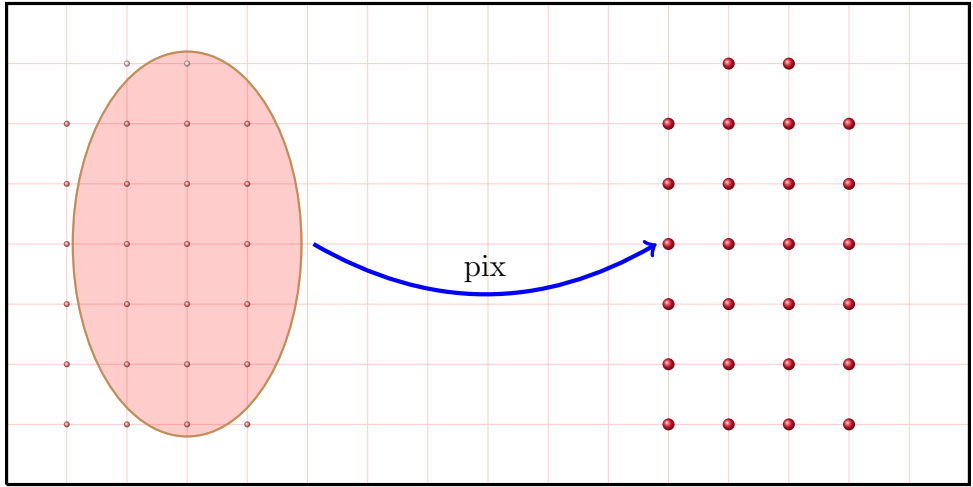

The following describes the role of the pixelization map 
in this paper.

\begin{definition}\label{Def:Dim}
	The \emph{macroscopic Minkowski dimension}
	of a set $F\subseteq\R^d$ is
	\[
		\DimM (F) := \limsup_{n\to\infty} n^{-1}\Log_+
		\left(\left|\pix(F)\cap B(0;2^n)\right|\right),
	\]
	where $|\,\cdots|$ denotes cardinality and $\Log_+(y):= \log_2(\max(y\,,2))$.
\end{definition}

\begin{remark}
	The right-hand side of \eqref{Def:Dim} coincides with  the 
	Barlow--Taylor \cite{BarlowTaylor} \emph{upper mass dimension}
	of the discrete set $\pix(F)\subseteq\Z^d$.
\end{remark}

The proof of the following elementary result
is left to the interested reader.

\begin{lemma}
	For every $A\subseteq\R^d$,
	\[
		\DimM(A) = \limsup_{n\to\infty} n^{-1}\log
		\left| \left\{ x\in B(0;2^n)\cap\Z^d:\ 
		Q(x)\cap A\neq\varnothing\right\}\right|,
	\]
	where $Q(x)$ was defined in \eqref{Q}.
\end{lemma}

Some of the elementary properties of $\DimM$ are listed below:
\begin{itemize}
\item If $A\subseteq B$ then $\DimM(A)\le\DimM(B)$;
\item If $A$ is a bounded set, then $\DimM(A)=0$;
\item $\DimM(\R^d)=\DimM(\Z^d)=d$.
\end{itemize}
The proof is omitted as it is easy to justify the preceding.

\subsection{Enumeration in shells}\label{subsec:shells}
There is a slightly different method of computing the macroscopic Minkowski
dimension of a set. With this aim in mind, define
\[
	\cS_0 := B(0;1),\qquad
	\cS_{n+1} := B(0;2^{n+1})\setminus B(0;2^n)\quad \text{for every integer
	$n \ge 0$.}
\]
One can think of $\cS_n$ as the $n$th \emph{shell} in $\Z^d$.

The following provides an alternative description of $\DimM (F)$.

\begin{proposition}\label{pr:Dim}
	For every $F\subseteq\R^d$,
	\[
		\DimM (F) := \limsup_{n\to\infty} n^{-1}\Log_+
		\left(\left|\pix(F)\cap\cS_n\right|\right).
	\]
\end{proposition}

Proposition \ref{pr:Dim} tells us that we can replace
$\pix(F)\cap B(0;2^n)$, in Definition \ref{Def:Dim}, by $\pix(F)\cap \cS_n$
without altering the formula for $\DimM(F)$.

\begin{proof}
	Our goal is to prove that $\DimM(F)=\delta(F)$, where
	\[
		\delta(F) := \limsup_{n\to\infty} n^{-1}\Log_+
		\left(\left|\pix(F)\cap\cS_n\right|\right).
	\]
	Since $\cS_n\subseteq B(0;2^n)$, the bound
	$\delta(F)\le\DimM(F)$ is immediate. We will establish the reverse inequality.
	
	The definition of $\delta(F)$ ensures that for every $\varepsilon\in(0\,,1)$ there
	exists an integer $N(\varepsilon)$ such that
	\[
		\left|\pix(F)\cap\cS_k\right|
		\le 2^{k\delta(F)(1+\varepsilon)}\qquad\text{for all $k\ge N(\varepsilon)$}.
	\]
	In particular, all $n\ge N(\varepsilon)$,
	\begin{align*}
		\left|\pix(F)\cap B(0;2^n)\right|
			&=\sum_{k=0}^n\left|\pix(F)\cap\cS_k\right|
			\le K(\varepsilon) + \sum_{k=N(\varepsilon)}^n 2^{k\delta(F)(1+\varepsilon)},\\
		&= O\left(2^{n\delta(F)(1+\varepsilon)}\right)\qquad[n\to\infty],
	\end{align*}
	where
	$K(\varepsilon) := \sum_{0\le k<N(\varepsilon)}|\cS_k|$ is finite and depends
	only on $(d\,,\varepsilon)$.  It follows from Definition \ref{Def:Dim}
	that $\DimM(F)\le\delta(F)(1+\varepsilon)$. This
	completes the proof since $\varepsilon\in(0\,,1)$ can be made to be as small as 
	one would like.
\end{proof}

\subsection{Boolean models}
In addition to the method of Proposition \ref{pr:Dim}, there is at least one
other useful method for computing the macroscopic Minkowski dimension of
a set. In contrast with the enumerative method of \S\ref{subsec:shells}, 
the method of this subsection is intrinsically probabilistic.

Let $\mathbf{p}:=\{p(x)\}_{x\in\Z^d}$ denote a collection of numbers in $(0\,,1)$,
and refer to the collection $\mathbf{p}$
as \emph{coverage probabilities},
in keeping with the literature on Boolean coverage processes \cite{Hall88}.

Let $\zeta:=\{\zeta(x)\}_{x\in\Z^d}$ denote a field of totally independent random
variables  that satisfy the following for all $x\in\Z^d$:
\[
	\P\{\zeta(x)=1\} = p(x) \quad
	\text{and}\quad\P\{\zeta(x)=0\}=1-p(x).
\]
By a \emph{Boolean model} in $\R^d$ with \emph{coverage probabilities}
$\mathbf{p}$ we mean the random set
\[
	\mathbf{B}(\mathbf{p}) := \bigcup_{\substack{x\in\Z^d:\\ \zeta(x)=1}} Q(x),
\]
where $Q(x)$ was defined earlier in \eqref{Q}.

If $A$ and $B$ are two subsets of $\R^d$, then we say that
$A$ is \emph{recurrent} for $B$ if $|A\cap B|=\infty$.
Equivalently, $A$ is recurrent for $B$ if $A\cap B\cap\cS_n\neq\varnothing$
for infinitely-many integers $n\ge0$. Clearly, if $A$ is recurrent
for $B$, then $B$ is also recurrent for $A$. Therefore, set recurrence is a
symmetric relation.

As the following result shows, it is  not hard to decide whether or not
a nonrandom Borel set $A\subseteq\R^d$
is recurrent for $\mathbf{B}(\mathbf{p})$.

\begin{lemma}\label{lem:BC}
	Let $A\subset\R^d$ be a nonrandom Borel set.
	Then,	
	\[
		\P\left\{ |A\cap\mathbf{B}(\mathbf{p})|=\infty\right\}=\begin{cases}
			1&\text{if $\sum_{x\in\pix(A)}
				p(x)=\infty$},\\
			0&\text{if $\sum_{x\in\pix(A)}
				p(x)<\infty$}.
		\end{cases}
	\]
\end{lemma}

Lemma \ref{lem:BC} is basically a reformulation
of the Borel--Cantelli lemma for independent
events. Therefore, we skip the proof. Instead,
let us mention the following,
more geometric, result which almost characterizes
recurrent sets in terms of their macroscopic Minkowski
dimension, in some cases.

\begin{proposition}\label{pr:Boolean}
	Suppose $\mathbf{p}$ has an \emph{index},
	\begin{equation}\label{alpha:p(x)}
		\text{\rm Ind}(\mathbf{p}) := -\lim_{|x|\to\infty}
		\frac{\log p(x)}{\log|x|}.
	\end{equation}
	Then for every nonrandom Borel set $A\subseteq\R^d$,
	\[
		\P\left\{ |A\cap\mathbf{B}(\mathbf{p})|=\infty\right\}=\begin{cases}
			1&\text{if $\DimM(A)>\text{\rm Ind}(\mathbf{p})$},\\
			0&\text{if $\DimM(A)<\text{\rm Ind}(\mathbf{p})$}.
		\end{cases}
	\]
\end{proposition}

We can compare this result to a similar result of Hawkes \cite{Hawkes81}  
about the hitting probabilities of the Mandelbrot fractal percolation. 
This comparison suggests that the Booelan models of this paper play
an analogous role in the theory of macroscopic fractals as does
fractal percolation in the better-studied theory of microscopic fractals.

\begin{OP}
	Is there a macroscopic analogue of the microscopic capacity theory of
	Peres \cite{Peres96a,Peres96}?
\end{OP}

\begin{proof}[Proof of Proposition \ref{pr:Boolean}]
	Let us consider the process $N_0,N_1,N_2,\ldots$,
	defined as
	\[
		N_n :=  |\mathbf{B}(\mathbf{p})\cap A\cap\cS_n|
		=\sum_{x\in \pix(A)\cap\cS_n}\zeta(x)
		\qquad[n\ge 0].
	\]
	Owing to \eqref{alpha:p(x)} and the definition of $\DimM$,
	\begin{equation}\label{E(Nn)}
		\limsup_{n\to\infty} n^{-1}\Log \E(N_n)=
		\DimM(A)-\text{\rm Ind}(\mathbf{p}).
	\end{equation}
	
	Suppose first that $\DimM(A)<\text{\rm Ind}(\mathbf{p})$.
	We may combine \eqref{E(Nn)} and Markov's inequality in
	order to see that
	$\sum_{n=1}^\infty\P \{ N_n >0 \}\le \sum_{n=1}^\infty\E(N_n)<\infty.$
	The Borel--Cantelli lemma then implies that with probability one
	$N_n=0$ for all but finitely-many integers $n$.
	That is, $|\mathbf{B}(\mathbf{p})\cap A)|<\infty$ a.s.\
	if $\DimM(A)<\text{\rm Ind}(\mathbf{p})$. This proves half of the proposition.
	
	For the remaining half let us assume that $\DimM(A)>\text{\rm Ind}(\mathbf{p})$,
	and  notice that
	$\Var(N_n) =\sum_{x\in \pix(A)\cap\cS_n}
	p(x)(1-p(x))\le\E(N_n).$
	Therefore,
	\begin{equation}\label{P:E(Nn)}
		\P\left\{ N_n \le \tfrac12\E(N_n) \right\}\le
		\P\left\{ |N_n -\E N_n| \ge \tfrac12\E(N_n)\right\}\le
		\frac{4\Var(N_n)}{|\E(N_n)|^2}
		\le \frac{4}{\E(N_n)},
	\end{equation}
	thanks to the Chebyshev's inequality. Because of \eqref{P:E(Nn)}
	there exists an infinite collection $\mathcal{N}$ of positive integers such that
	\[
		n^{-1}\Log \E(N_n) \to  \DimM(A)-\text{\rm Ind}(\mathbf{p})>0
		\qquad\text{as $n$ approaches infinity in $\mathcal{N}$.} 
	\]
	This fact, and \eqref{E(Nn)}, together imply that
	$\sum_{n\in\mathcal{N}}\P\{N_n\le\frac12\E(N_n)\}<\infty$, and hence
	\[
		\DimM(\mathbf{B}(\mathbf{p})\cap A) = 
		\limsup_{n\to\infty} n^{-1}\Log N_n\ge
		\lim_{\substack{n\to\infty:\\n\in\mathcal{N}}}
		n^{-1}\Log N_n \ge \DimM(A)-\text{\rm Ind}(\mathbf{p})>0,
	\]
	almost surely. This completes the proof.
\end{proof}

\begin{remark}
	A quick glance at the proof shows that the independence
	of the $\zeta$'s was needed only to show that
	\begin{equation}\label{Var:cond}
		\Var(N_n)=O(\E(N_n))\qquad\text{as
		$n\to\infty$.}
	\end{equation}
	Because $\Var(N_n) = \sum_{x,y\in\pix(A)\cap\cS_n}
	\P\{\zeta(x)=\zeta(y)=1\}$, \eqref{Var:cond}
	continues to hold if the independence
	of the $\zeta$'s is relaxed to a condition such as the following:
	There exists finite and positive constants $c$ and $K$ such that
	\[
		\P[\zeta(x)=1\mid\zeta(y)=1] \le c\P\{\zeta(x)=1\}\quad
		\text{whenever $\|x\|\wedge\|y\|\ge K$}.
	\]
\end{remark}

We highlight the power of Proposition \ref{pr:Boolean} by using it to
give a quick computation of $\DimM(A\cap\mathbf{B}(\mathbf{p}))$.

\begin{corollary}\label{cor:Boolean}
	If $A\subseteq\R^d$ denotes a nonrandom Borel set, then
	\[
		\DimM\left(A\cap\mathbf{B}(\mathbf{p})\right) =
		\DimM(A) - \text{\rm Ind}(\mathbf{p})\qquad\text{a.s.}
	\]
\end{corollary}

\begin{corollary}\label{co:dim}
	$\DimM(\mathbf{B}(\mathbf{p}))=d-\text{\rm Ind}(\mathbf{p})\qquad\text{%
	a.s.}$
\end{corollary} 

\noindent
\begin{figure}[!ht]\centering
	\begin{subfigure}[b]{0.4\textwidth}\centering
		\fbox{\includegraphics[width=\textwidth]{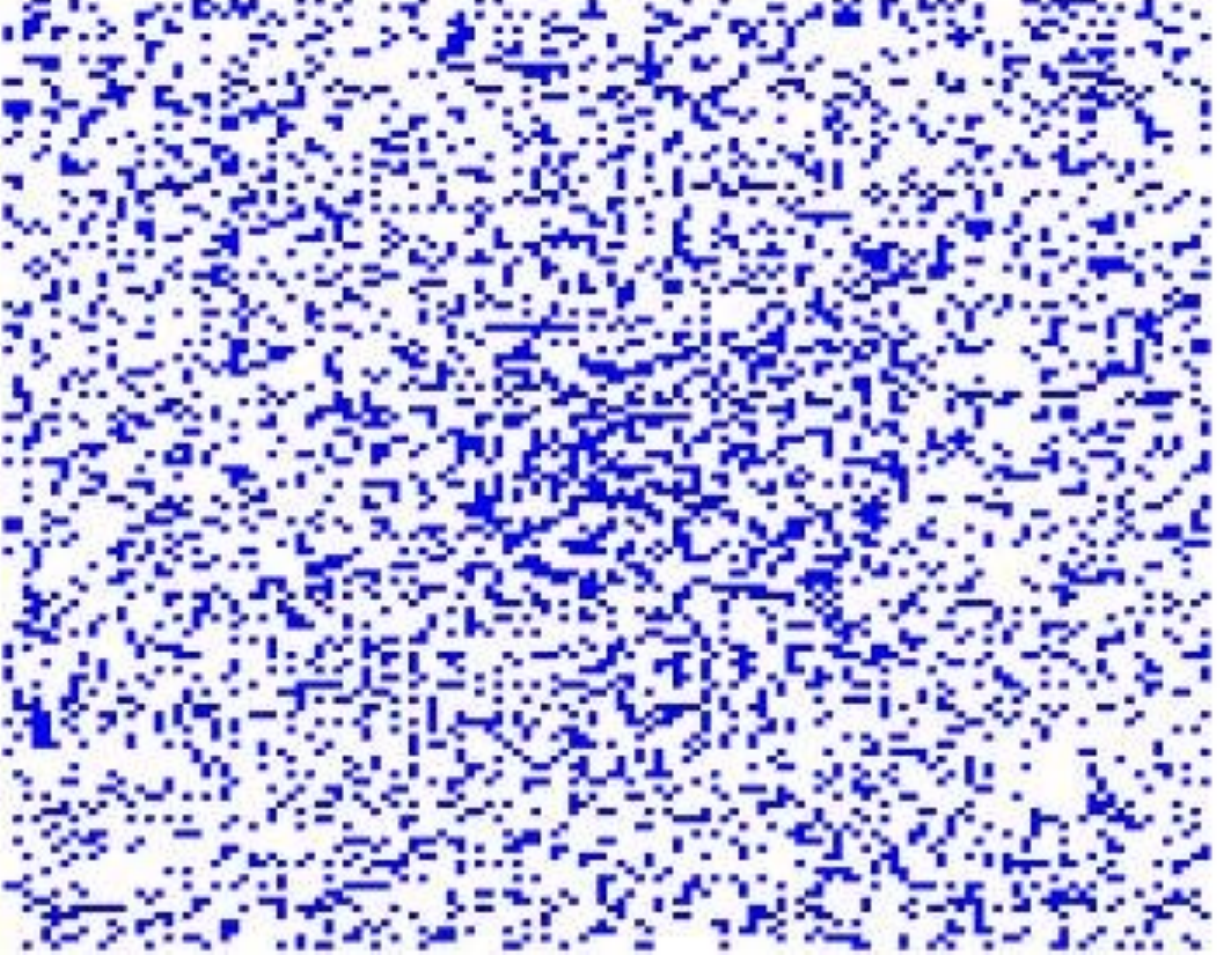}}
		\caption{$\text{Ind}({\mathbf{p}})=0.3$, 
			$\DimM(\mathbf{B}(\mathbf{p}))=1.7$}
	\end{subfigure}\hfil
	\begin{subfigure}[b]{0.4\textwidth}\centering
		\fbox{\includegraphics[width=.975\textwidth]{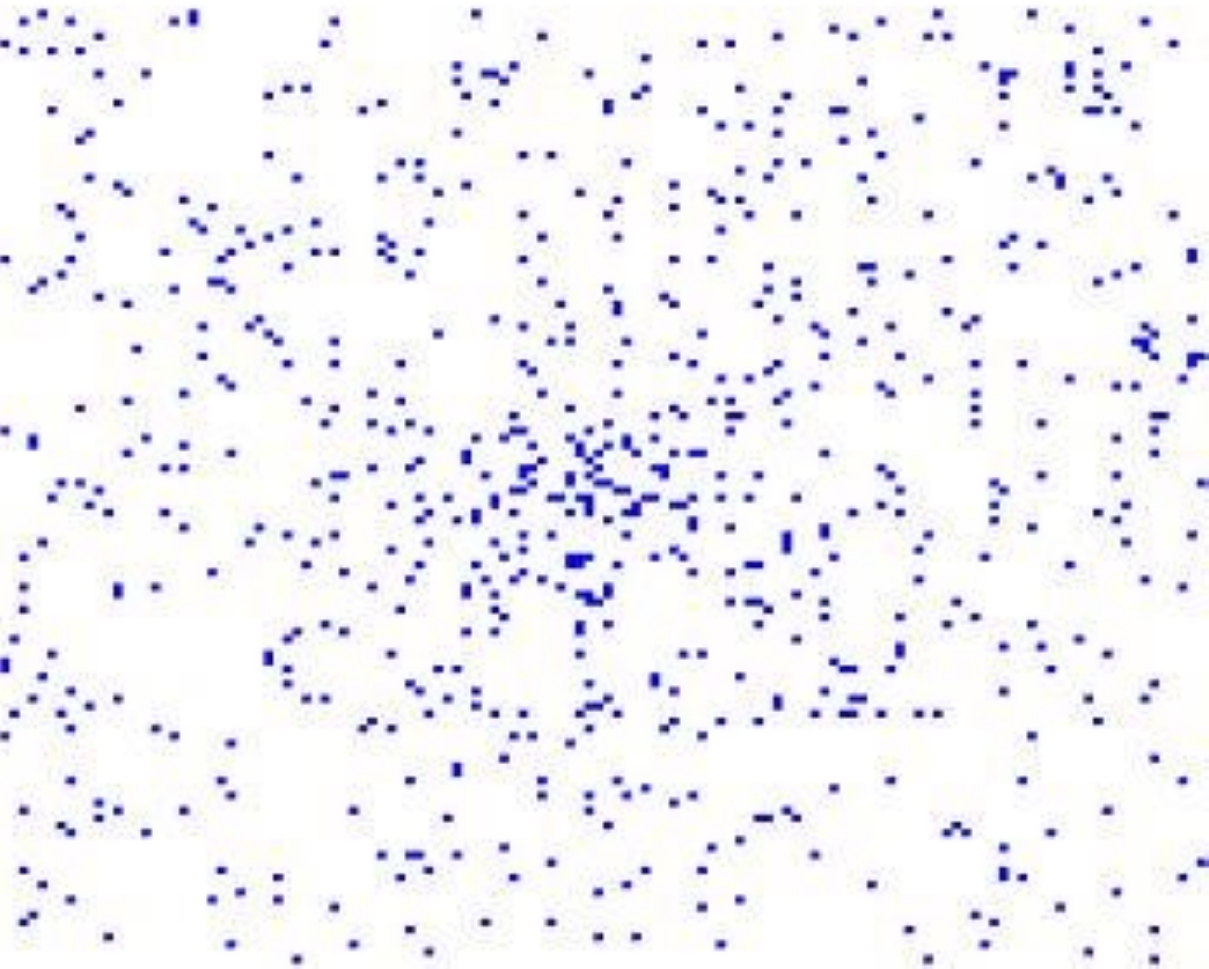}}
		\caption{$\text{Ind}({\mathbf{p}})=0.7$, 
			$\DimM(\mathbf{B}(\mathbf{p}))=1.3$}
	\end{subfigure}
	
	\caption{A simulation of two Boolean models.
		Corollary \ref{co:dim} ensures that
		the Minkowski dimensions of the two figures
		are respectively $1.7$ (left) and $1.3$ (right).}
	\label{fig:Bool}
\end{figure}

Because $\DimM(\R^d)=d$,
the following is an immediate consequence of Corollary \ref{cor:Boolean}.
Therefore, it remains to establish Corollary \ref{cor:Boolean}.
The proof uses a variation of an elegant ``replica argument'' that was introduced by
Peres \cite{Peres96} in the context of [microscopic] Hausdorff dimension of
fractal percolation processes.

\begin{proof}[Proof of Corollary \ref{cor:Boolean}]
	Let $\mathbf{B}'(\mathbf{p}')$ be an independent Boolean model
	with coverage probabilities $\mathbf{p}'=\{p'(x)\}_{x\in\Z^d}$ that have an index
	$\text{\rm Ind}(\mathbf{p}')$.
	Define $q(x):=p(x)\times p'(x)$ for all $x\in\Z^d$.
	It is then easy to see
	that $\mathbf{C}(\mathbf{q}):=
	\mathbf{B}'(\mathbf{p}') \cap \mathbf{B}(\mathbf{p})$ is a Boolean model with
	coverage probabilities $\mathbf{q}=\{q(x)\}_{x\in\Z^d}$.
	Since $\text{\rm Ind}(\mathbf{q})=\text{\rm Ind}(\mathbf{p})+
	\text{\rm Ind}(\mathbf{p}')$, Proposition
	\ref{pr:Boolean} implies that
	\[
		\P\left\{ |A\cap\mathbf{C}(\mathbf{q})| =\infty\right\}
		=\begin{cases}
			1&\text{if $\text{\rm Ind}(\mathbf{p})+
				\text{\rm Ind}(\mathbf{p}')<\DimM(A)$},\\
			0&\text{if $\text{\rm Ind}(\mathbf{p})+
				\text{\rm Ind}(\mathbf{p}')>\DimM(A)$}.
		\end{cases}
	\]
	At the same time, one can apply Proposition \ref{pr:Boolean}
	conditionally in order to see that almost surely,
	\begin{align*}
		\P\left(|A\cap\mathbf{C}(\mathbf{q})| =
			\infty\mid \mathbf{B}(\mathbf{p})\right)
			&=\P\left(|A\cap\mathbf{B}(\mathbf{p})\cap\mathbf{B}'(\mathbf{p'})|
			=\infty\mid \mathbf{B}(\mathbf{p})\right)\\
		&=\begin{cases}
				1&\text{if $\DimM\left(A\cap\mathbf{B}(\mathbf{p})\right)
					>\text{\rm Ind}(\mathbf{p'})$},\\
				0&\text{if $\DimM\left(A\cap\mathbf{B}(\mathbf{p})\right)
					<\text{\rm Ind}(\mathbf{p'})$}.
			\end{cases}
	\end{align*}
	A comparison of the preceding two displays yields
	the following almost sure assertions:
	\begin{enumerate}
	\item If $\text{\rm Ind}(\mathbf{p})+ \text{\rm Ind}(\mathbf{p}')<
		\DimM(A)$,
		then $\DimM\left(A\cap\mathbf{B}(\mathbf{p})\right)\ge
		\text{\rm Ind}(\mathbf{p'})$
		a.s.; and
	\item If $\text{\rm Ind}(\mathbf{p})+ \text{\rm Ind}(\mathbf{p}')>\DimM(A)$,
		then $\DimM\left(A\cap
		\mathbf{B}(\mathbf{p})\right)\le\text{\rm Ind}(\mathbf{p'})$
		a.s.
	\end{enumerate}
	Since $\mathbf{p}'$ can have any arbitrary index $\text{\rm Ind}(\mathbf{p}')>0$
	that one wishes, the corollary follows.
\end{proof}

\section{Transient L\'evy processes}

Let $X:=\{X_t\}_{t\ge0}$ be a  L\'evy process on $\R^d$.
That is, $X$ is a strong Markov process that has c\`adl\`ag paths,
takes values in $\R^d$, $X_0=0$, and $X$ has stationary and independent increments. 
See,  for example, Bertoin \cite{Bertoin96} for a pedagogic account.
In this section we assume that $X$ is transient and compute the macroscopic dimension
of the range $\cR_{_X}$ of $X$, where we recall the range is the following random set:
\[
	\cR_{_X} := \bigcup_{t\ge0}\{X_t\}.
\]

\subsection{The potential measure}
Let $\U_{_X}$ denote the potential measure of $X$; that is,
\begin{equation}\label{Pot:Meas}
	\U_{_X}(A) := \int_0^\infty \P\{X_t\in A\}\,\d t=
	\E\int_0^\infty \1_A(X_t)\,\d t.
\end{equation}
Throughout we assume that $X$ is transient; equivalently, $\U_{_X}$ is
a Radon measure. The following shows that the macroscopic Minkowski
dimension of the range of $X$ is linked intimately to the potential measure
of $X$.

\begin{theorem}\label{th:Range:Levy}
	With probability one,
	\[
		\DimM(\cR_{_X}) = \inf\left\{\alpha>0:\
		\int_{\R^d}\frac{\U_{_X}(\d x)}{1+|x|^\alpha}<\infty\right\}.
	\]
\end{theorem}
Theorem \ref{th:Range:Levy:FT} contains an alternative formula
for $\DimM(\cR_{_X})$, in terms of the L\'evy exponent of $X$, which
is reminiscent of an old formula of Pruitt \cite{Pruitt69} for the [micrsoscopic]
Hausdorff dimension of $\cR_{_X}$. We refer to Ref.'s \cite{KX05,KX09,KXZ} 
for more recent developments on micrsoscopic fractal properties of L\'evy 
processes, based on potential theory of additive L\'evy processes.

\begin{example}
	Consider the case that $X:=\{X_t\}_{t\ge0}$ is a symmetric
	$\beta$-stable process on $\R^d$ for some $0<\beta\le 2$.
	Transience is equivalent to the condition $\alpha<d$.
	This condition is known
	to imply that $\U_{_X}(\d x)/\d x\propto \|x\|^{-d+\beta}$
	for all $x\in\R^d\setminus\{0\}$ \cite{Bertoin96,PS71}.
	Therefore,
	$\int_{\R^d} (1+|x|^\alpha)^{-1}\,\U_{_X}(\d x)<\infty$ iff
	$\int_{|x|>1} |x|^{-\alpha-d+\beta}\,\d x<\infty$ iff
	$\alpha>\beta$. Theorem \ref{th:Range:Levy} then implies that
	$\DimM(\cR_{_X})=\beta$ a.s. This fact is essentially
	due to Barlow and Taylor \cite{BarlowTaylor}.
\end{example}

\begin{remark}
	Recall that the measure $\U_{_X}$ is finite because $X$ is transient.
	As a result, $\int_{\R^d}(1+|x|^\alpha)^{-1}\,\U_{_X}(\d x)$
	converges iff $\int_{|x|>1}|x|^{-\alpha}\,\U_{_X}(\d x)<\infty$.
	One can then deduce from this fact,  from the definition
	\eqref{Pot:Meas} of $\U_{_X}$,
	and from Theorem \ref{th:Range:Levy}  that
	\[
		\DimM(\cR_{_X}) = \inf\left\{\alpha>0:\
		\int_0^\infty \E\left(|X_t|^{-\alpha}; |X_t|>1\right)\d t
		<\infty\right\}\qquad\text{a.s.}
	\]
	This is the macroscopic analogue of a result of Pruitt \cite[p.\ 374]{Pruitt69}.
\end{remark}

\begin{OP}
	It is natural to ask if there is a nice formula for
	$\DimM(A\cap\cR_{_X})$ when $A\subseteq\R^d$ is Borel and
	nonrandom. We do not have an answer to this question when $A$ is
	not  ``macroscopically self-similar.''
\end{OP}

The proof of Theorem \ref{th:Range:Levy} hinges on a few
prefatory technical results. The first is a more-or-less well-known
set of bounds on the potential measure of a ball.

\begin{lemma}\label{lem:U1}
	For every $x\in\R^d$ and $r>0$,
	\[
		\U_{_X}(B(x;r)) \le \U_{_X}(B(0;2r))\cdot\P\left\{ \overline{\cR_{_X}}
		\cap B(x;r)\neq\varnothing\right\}.
	\]
\end{lemma}

\begin{proof}
	Let $\inf\varnothing:=\infty$, and consider the stopping time
	\begin{equation}\label{T}
		T(x;r) := \inf\{t\ge0: X_t\in B(x;r) \}.
	\end{equation}
	We can write $\U_{_X}(B(x;r))$ in the following equivalent form:
	\begin{equation}\label{U}
		\E\left(\int_0^\infty\1_{B(x-X_{T(x;r)},r)}
		\left(X_{t+T(x;r)}-X_{T(x;r)}\right)\d t\cdot\1_{\{T(x;r)<\infty\}}\right).
	\end{equation}
	Since $|X_{T(x;r)}-x|<r$ a.s.\ on the event $\{T(x;r)<\infty\}$,
	the triangle inequality implies that
	$B(x-X_{T(x;r)}\,,r)\subseteq B(0;2r)$ a.s.\ on $\{T(x;r)<\infty\}$,
	and hence
	\[
		\U_{_X}(B(x;r))\le\U_{_X}(B(0;2r))\cdot\P\{T(x;r)<\infty\}.
	\]
	This is another way to state the lemma.
\end{proof}

The next result is a standard upper bound on the hitting probability
of a ball.

\begin{lemma}\label{lem:U2}
	For every $x\in\R^d$ and $r>0$,
	\[
		\U_{_X}(B(x;2r)) \ge \U_{_X}(B(0;r))\cdot\P\left\{ \overline{\cR_{_X}}
		\cap B(x;r)\neq\varnothing\right\}.
	\]
\end{lemma}

\begin{proof}
	By the triangle inequality,
	$B(x-X_{T(x;r)}\,,2r)\supset B(0;r)$
	almost surely on the event $\{T(x;r)<\infty\}$, where $T(x;r)$ was
	defined in \eqref{T}. Therefore, we apply \eqref{U} together
	with the strong Markov property in order to see that
	\[
		\U_{_X}(B(x;2r)) \ge \U_{_X}(B(0;r))\cdot\P\{T(x;r)<\infty\}.
	\]
	This is another way to write the lemma.
\end{proof}

The following is a ``weak unimodality'' result for
the potential measure.

\begin{lemma}\label{lem:U3}
	$\U_{_X}(B(x;r))\le 4^d\U_{_X}(B(0;r))$ for all $x\in\R^d$ and $r>0$.
\end{lemma}

\begin{proof}
	The proof will use the following elementary covering property of Euclidean spaces:
	For every $x\in\R^d$ and $r>0$ there exist points  $y_1,\ldots,y_{4^d}\in B(x;r)$
	such that $B(x;r) =\cup_{1\le i\le 4^d} B(y_i\,,r/2).$
	This leads to the following ``volume-doubling'' bound: For all $r>0$ and $x\in\R^d$,
	\begin{equation}\label{eq:Vol:Dbl}
		\U_{_X}(B(x;r)) \le 4^d\sup_{y\in B(x,r)}\U_{_X}(B(y;r/2)).
	\end{equation}
	This inequality yields the lemma since
	$\U_{_X}(B(y;r/2))\le \U_{_X}(B(0;r))$ for all $y\in\R^d$
	and $r>0$, thanks to Lemma \ref{lem:U1}.
\end{proof}

The next result presents bounds for the probability that the
pixelization of the range of $X$ hits singletons. Naturally, both
bounds are in terms of the potential measure of $X$.

\begin{lemma}\label{lem:U4}
	There exist finite constants $c_2>1>c_1>0$ such that,
	for all $x\in\Z^d$,
	\[
		c_1\U_{_X}(Q(x))\le\P\left\{ x\in\pix\left(\cR_{_X}\right)\right\}\le
		c_2  \U_{_X}(B(x;2)).
	\]
\end{lemma}

\begin{proof}
	Since the process $X$ is c\`adl\`ag, the difference between
	$\cR_{_X}$ and its closure is a.s.\ denumerable, and
	hence
	\[
		\P\left\{ x\in\pix\left(\cR_{_X}\right)\right\}
		=\P\left\{ x\in\pix\left(\overline{\cR_{_X}}\right)\right\}
		=\P\left\{ \overline{\cR_{_X}}\cap Q(x)\neq\varnothing\right\},
	\]
	for all $x\in\Z^d$. Let $y_i:=x_i+\frac12$ for $1\le i\le d$
	and recall that $Q(x)=B(y;1/2)$ in order to deduce from
	Lemmas \ref{lem:U1} and \ref{lem:U2} that
	\begin{equation}\label{eq:prob:hit}
		\frac{\U_{_X}(Q(x))}{\U_{_X}(B(0;1))}=
		\frac{\U_{_X}(B(y;1/2))}{\U_{_X}(B(0;1))}\le
		\P\left\{ x\in\pix\left(\cR_{_X}\right)\right\}
		\le \frac{\U_{_X}(B(y;1))}{\U_{_X}(B(0;1/2))}.
	\end{equation}
	The denominators are strictly positive because $X$ is c\`adl\`ag
	and $B(0;1/2)$ is open in $\R^d$; and they are finite
	because of the transience of $X$.
	Because $B(y;1)\subseteq B(x;2)$, \eqref{eq:prob:hit} completes the proof.
\end{proof}

The following lemma is the final technical result of this section. It
presents an upper bound for the probability that the range
of $X$ simultaneously intersects two given balls.

\begin{lemma}\label{lem:U5}
	For all $x,y\in\R^d$ and $r>0$,
	\begin{align*}
		&\P\left\{\overline{\cR}_X\cap B(x;r)\neq\varnothing\,,
			\overline{\cR_{_X}}\cap B(y;r)\neq\varnothing \right\}\\
		&\hskip1in\le \frac{\U_{_X}(B(x;2r))}{\U_{_X}(B(0;r))}\cdot
			\frac{\U_{_X}(B(y-x;4r))}{\U_{_X}(B(0;2r))}
			+\frac{\U_{_X}(B(y;2r))}{\U_{_X}(B(0;r))}\cdot
			\frac{\U_{_X}(B(x-y;4r))}{\U_{_X}(B(0;2r))}.
	\end{align*}
\end{lemma}

\begin{proof}
	Let us recall the stopping times $T(x;r)$ from \eqref{T}. First one notices that
	\begin{align*}
		\P\left\{ T(x;r)\le T(y;r)<\infty\right\}
		&=\P\left\{ T(x;r)<\infty, \,\exists s\ge0: X_{s+T(x;r)} - X_{T(x;r)}\in
			B(y- X_{T(x;r)};r)\right\}\\
		&\le \P\{ T(x;r)<\infty\}\cdot\P\{ T(y-x;2r)<\infty\}\\
		&= \P\left\{ \overline{\cR_{_X}}\cap B(x;r)\neq\varnothing\right\}
			\cdot \P\left\{ \overline{\cR_{_X}}\cap B(y-x;2r)\neq\varnothing\right\},
	\end{align*}
	owing to the strong Markov property and the fact that
	$B(y - X_{T(x;r)}\,,r)\subseteq B(y-x;2r)$ a.s.\ on $\{T(x;r)<\infty\}$
	[the triangle inequality]. By replacing also the roles of $x$ and $y$ and appealing
	to the subadditivity of probabilities, one can deduce from the preceding  that
	\begin{align*}
		&\P\left\{\overline{\cR}_X\cap B(x;r)\neq\varnothing\,,
			\overline{\cR_{_X}}\cap B(y;r)\neq\varnothing \right\}\\
		&\hskip1in\le \P\left\{ \overline{\cR_{_X}}\cap B(x;r)\neq\varnothing\right\}
			\cdot \P\left\{ \overline{\cR_{_X}}\cap B(y-x;2r)\neq\varnothing\right\}\\
		&\hskip2.5in + \P\left\{ \overline{\cR_{_X}}\cap B(y;r)\neq\varnothing\right\}
			\cdot \P\left\{ \overline{\cR_{_X}}\cap B(x-y;2r)\neq\varnothing\right\}.
	\end{align*}
	An appeal to Lemma \ref{lem:U2} completes the proof.
\end{proof}

With the requisite material for the proof of Theorem \ref{th:Range:Levy}
under way, we are ready for the following.

\begin{proof}[Proof of Theorem \ref{th:Range:Levy}]
	The strategy of the proof is to verify that $\DimM(\cR_{_X})=\alpha_c$ a.s.,
	where
	\[
		\alpha_c := \inf\left\{\alpha>0:\ \int_{\R^d}
		\frac{\U_{_X}(\d x)}{1+|x|^\alpha}<\infty\right\}.
	\]
	Let us begin by making some real-variable observations.
	First, let us note that because $\U_{_X}$ is a finite measure [by transience],
	\[
		\sum_{n=1}^\infty 2^{-n\alpha}\U_{_X}(\cS_n)
		=\sum_{n=1}^\infty 2^{-n\alpha}\int_{\cS_n}\U_{_X}(\d x)
		\asymp \int_{|x|>1} \frac{\U_{_X}(\d x)}{|x|^\alpha}
		\asymp\int_{\R^d}\frac{\U_{_X}(\d x)}{1+|x|^\alpha}.
	\]
	Therefore,
	\[
		\alpha_c =
		\inf\left\{\alpha>0:\ \sum_{n=1}^\infty 2^{-n\alpha}
		\U_{_X}(\cS_n) <\infty\right\}.
	\]
	By the definition of $\alpha_c$, if $0<\alpha<\alpha_c$, then
	$\sum_n 2^{-n\alpha}\U_{_X}(\cS_n)=\infty$; as a result, 
	\[
		\limsup_{n\to\infty} 2^{-\beta n}\U_{_X}(B(0;2^n))\ge
		 \limsup_{n\to\infty}2^{-\beta n}
		\U_{_X}(\cS_n) =\infty,
	\]
	whenever $0<\beta<\alpha$. On the other hand,
	 if $\beta>\alpha_c$, then $\limsup_{n\to\infty} 2^{-\beta n}\U_{_X}(\cS_n)<\infty$,
	 and hence
	 \[
	 	\U_{_X}(B(0;2^n))=\sum_{k=0}^n\U_{_X}(\cS_k)=O(2^{\beta n})
		\qquad\text{as  $n\to\infty$.}
	\]
	These remarks together show the following alternative representation
	 of $\alpha_c$:
	 \begin{equation}\label{red2}
	 	\alpha_c = \limsup_{n\to\infty} n^{-1}\Log \U_{_X}(B(0;2^n))
		=  \limsup_{n\to\infty} n^{-1}\Log \U_{_X}(\cS_n).
	 \end{equation}
	 Now we begin the bulk of the proof.
	
	Because $X$ has c\`adl\`ag sample functions,
	Lemma \ref{lem:U4} and \eqref{red2} together
	imply that for all $n\ge2$,
	\[
		\E\left|\pix\left(\cR_{_X}\right)\cap B(0;2^n) \right|
		\lesssim\sum_{x\in B(0;2^n)}\U_{_X}(B(x;2))
		\lesssim \U_{_X}(B(0;2^{n+1}))
		\le 2^{n(1+o(1))\DimM(\cR_{_X})},
	\]
	as $n\to\infty$.
	Therefore, the Chebyshev inequality implies that
	\[
		\sum_{n=1}^\infty\P\left\{ \left|\pix
		\left(\cR_{_X}\right)\cap B(0;2^n) \right|>
		2^{n\theta}\right\} <\infty
		\qquad\text{for all $\theta>\alpha_c$}.
	\]
	An application of the Borel--Cantelli lemma yields
	$\DimM(\cR_{_X})\le\alpha_c$ a.s.,
	which implies a part of the assertion of the theorem. 
	
	For the next part,
	let us  begin with the following consequence of Lemma \ref{lem:U4}:
	\begin{equation}\label{E:pix:LB}
		\E \left|\pix\left(\cR_{_X}\right)\cap B(0;2^n)\right|
		\gtrsim \sum_{x\in B(0;2^n)}\U_{_X}(Q(x)) 
		\asymp \U_{_X}(B(0;2^n)).
	\end{equation}
	Next, we estimate the second moment of the same random variable as follows:
	\begin{align*}
		\E\left(\left|\pix\left(\cR_{_X}\right)\cap
			B(0;2^n) \right|^2\right)
			&\le \sum_{x,y\in B(0;2^n)}\P\left\{\overline{\cR}_X\cap B(x;1)\neq\varnothing\,,
			\overline{\cR_{_X}}\cap B(y;1)\neq\varnothing \right\}\\
		&\le \sum_{x,y\in B(0;2^n)}\frac{\U_{_X}(B(x;2))}{\U_{_X}(B(0;1))}
			\cdot\frac{\U_{_X}(B(y-x;4))}{\U_{_X}(B(0;2))}\\
		&\hskip1in +\sum_{x,y\in B(0;2^n)}\frac{\U_{_X}(B(y;2))}{\U_{_X}(B(0;1))}
			\cdot\frac{\U_{_X}(B(x-y;4))}{\U_{_X}(B(0;2))};
	\end{align*}
	see Lemma \ref{lem:U5} for the final inequality. Since $B(0;2^n)\ominus B(0;2^n)
	=B(0\,,2^{n+1})$, it follows that
	\begin{align*}
		\E\left(\left|\pix\left(\cR_{_X}\right)\cap B(0;2^n) \right|^2\right)
			&\le 2\sum_{x\in B(0;2^n)}\frac{\U_{_X}(B(x;2))}{\U_{_X}(B(0;1))}\ \cdot
			\sum_{w\in B(0;2^{n+1})}\frac{\U_{_X}(B(w;4))}{\U_{_X}(B(0;2))}\\
		&\le K\U_{_X}(B(0;2^{n+1}))\cdot \U_{_X}(B(0;2^{n+2})\\
		&\le 4^{3d} K[\U_{_X}(B(0;2^n))]^2,
	\end{align*}
	where $K:= 2[\U_{_X}(B(0;1))\U_{_X}(B(0;2))]^{-1}$ and the last line follows from
	\eqref{eq:Vol:Dbl}.
	Therefore, the Paley--Zygmund inequality and \eqref{E:pix:LB} together imply that
	\begin{align*}
		\P\left\{\left|\pix\left(\cR_{_X}\right)\cap B(0;2^n) \right|
			>\frac{c_1}{2}\U_{_X}(B(0;2^n))\right\}
			&\ge \frac{\left( \E\left|\pix\left(\cR_{_X}\right)
			\cap B(0;2^n)\right|\right)^2}{4\E\left(
			\left|\pix\left(\cR_{_X}\right)\cap B(0;2^n) \right|^2\right)}\\
		&\gtrsim 1,
	\end{align*}
	uniformly in $n$. The preceding and \eqref{red2} together imply that
	$\DimM(\cR_{_X})\ge\alpha_c$ a.s. This verifies the theorem
	since the other bound was verified earlier in the proof.
\end{proof}

\subsection{Fourier analysis}

It is well-known that the law of $X$
is determined by a socalled \emph{characteristic exponent} $\Psi_{_X}:\R^d\to\C$,
which can be defined via
$\E\exp(iz\cdot X_t)=\exp(-t\Psi_{_X}(z))$ for all $t\ge0$ and $z\in\R^d$.
In particular, one can prove from this that $\Psi_{_X}(z)\neq 0$ for almost all
$z\in\R^d$. This fact is used tacitly in the sequel.

We frequently use the well-known fact that
$\Re\Psi_{_X}(z)\ge0$ for all $z\in\R^d$. To see this fact,
let $X'$ be an independent
copy of $X$ and note that $t\mapsto X_t-X_t'$ is a L\'evy process
with characteristic exponent $2\Re\Psi_{_X}$. Since $X_1-X'_1$ is a symmetric
random variable, one can conclude the mentioned fact that $\Re\Psi_{_X}\ge0$.

Port and Stone \cite{PS71} have proved,
among other things, that the transience of $X$ is equivalent to
the convergence of the integral
\[
	I(\Psi_{_X}):=\int_{\|z\|\le 1}\Re\left(\frac{1}{\Psi_{_X}(z)}\right)\d z;
\]
see also \cite{Bertoin96}.
The following shows that the macroscopic dimension of the
range of $X$ is determined by the strength by which the
Port--Stone integral $I(\Psi_{_X})$ converges.

\begin{theorem}\label{th:Range:Levy:FT}
	With probability one,
	\[
		\DimM(\cR_{_X}) = \inf\left\{\alpha>0:\
		\int_{\|z\|\le1}\Re\left( \frac{1}{\Psi_{_X}(z)}\right)\frac{\d z}{%
		\|z\|^{d-\alpha}}<\infty\right\}.
	\]
\end{theorem}

The proof of Theorem \ref{th:Range:Levy:FT} hinges
on a calculation from classical Fourier analysis.
From now on, $\widehat{h}$ denotes the Fourier transform
of a locally integrable function $h:\R^d\to\R$, normalized so that
\[
	\widehat{h}(z) = \int_{\R^d}\e^{iz\cdot x}h(x)\,\d x
	\qquad\text{for all $z\in\R^d$ and $h\in L^1(\R^d$)}.
\]
As is done customarily, we
let $K_\nu$ denote the modified Bessel function  [Macdonald function]
of the second kind.

\begin{lemma}\label{lem:FT}
	Choose and fix $\alpha>0$ and define
	$f(x) := (1+\|x\|^2)^{-\alpha/2}$ for all $x\in\R^d$. Then,
	the Fourier transform of $f$ is
	\[
		\widehat{f}(z) = c_{d,\alpha}\cdot
		\frac{K_{(d-\alpha)/2}(\|z\|)}{\|z\|^{(d-\alpha)/2}}
		\qquad[z\in\R^d],
	\]
	where $0<c_{d,\alpha}<\infty$ depends only on $(d\,,\alpha)$.
\end{lemma}

\begin{proof}
	This is undoubtedly well known; the proof hinges on a simple
	abelian trick that can be included with little added
	effort.
	
	For all $x\in\R^d$ and $\theta>0$,
	\[
		\int_0^\infty \e^{-t(1+\|x\|^2)} t^{\theta-1}\,\d t = \frac{
		\Gamma(\theta)}{(1+\|x\|^2)^\theta}.
	\]
	Therefore, for every rapidly decreasing test function $\varphi:\R^d\to\R$,
	\begin{align*}
		\int_{\R^d}\frac{\varphi(x)}{(1+\|x\|^2)^\theta}\,\d x
			&= \frac{1}{\Gamma(\theta)}
			\int_{\R^d}\varphi(x)\,\d x\int_0^\infty\d t\
			\e^{-t(1+\|x\|^2)}t^{\theta-1}\\
		&=\frac{1}{\Gamma(\theta)}\int_0^\infty \e^{-t}t^{\theta-1}\,\d t
			\int_{\R^d}\varphi(x) \e^{-t\|x\|^2}\,\d x.
	\end{align*}
	Since
	\[
		\int_{\R^d} \varphi(x)\e^{-t\|x\|^2}\,\d x = (4\pi t)^{-d/2}
		\int_{\R^d}\overline{\widehat{\varphi}(z)}
		\exp\left( - \frac{\|z\|^2}{4t}\right)\d z,
	\]
	it follows that
	\begin{align*}
		\int_{\R^d}\frac{\varphi(x)}{(1+\|x\|^2)^\theta}\,\d x
			&=\frac{1}{b}\int_{\R^d}\overline{\widehat{\varphi}(z)}\,\d z
			\int_0^\infty \frac{\d t}{t^{(d/2)-\theta+1}}
			\exp\left( -t- \frac{\|z\|^2}{4t}\right)\d z\\
		&= \frac{1}{c}\int_{\R^d}\overline{\widehat{\varphi}(z)}\,
			\frac{K_{(d/2)-\theta}(\|z\|)}{\|z\|^{(d/2)-\theta}}\,\d z,
	\end{align*}
	where $b:=(4\pi)^{d/2}\Gamma(\theta)$
	and $c:=(4\pi)^{d/2}\Gamma(\theta)2^{-1-(d/2)+\theta}$.
	This proves the result, after we set $\theta:=\alpha/2$.
\end{proof}

\begin{proof}[Proof of Theorem \ref{th:Range:Levy:FT}]
	It is not hard to check
	(see, for example, Port and Stone \cite{PS71})
	that $\widehat{\U}_{_X}(z) = 1/\Psi_{_X}(z)$ for almost all $z\in\R^d$.
	Because $\Re(1/\Psi_{_X}(z))=\Re\Psi_{_X}(z)/|\Psi_{_X}(z)|^2>0$ a.e.,
	Lemma \ref{lem:FT} and a suitable form of
	the Plancherel's theorem together imply that
	\[
		\int_{\R^d}\frac{\U_{_X}(\d x)}{1+|x|^\alpha}
		\asymp\int_{\R^d}\frac{\U_{_X}(\d x)}{(1+|x|^2)^{\alpha/2}}
		\propto\int_{\R^d}\Re\left( \frac{1}{\Psi_{_X}(z)}\right)
		\frac{K_{(d-\alpha)/2}(\|z\|)}{\|z\|^{(d-\alpha)/2}}\,\d z
		:= T_1 + T_2,
	\]
	where $T_1$ denotes the preceding integral with domain of integration
	restricted to $\{z\in\R^d:\,|\Psi(z)| <1\}$ and $T_2$ is the same
	integral over $\{z\in\R^d:\,|\Psi(z)| \ge1\}$.
	
	A standard application of  Laplace's method shows that for all $R>0$ 
	there exists
	a finite $A>1$ such that 
	\[
		\frac{\e^{-w}}{A\sqrt w}\le
		K_\nu(w)\le \frac{A\e^{-w}}{\sqrt w},
	\]
	whenever $w>R$.
	And one can check directly that  for all $R>0$ we can find
	a finite $B>1$  such that
	\[
		B^{-1}w^{-\nu}\le K_\nu(w)\le B w^{-\nu}
		\qquad\text{whenever $0<w<R$.}
	\]
	Since $\Psi_{_X}:\R^d\to\C$ is a continuous function that vanishes at the origin,
	$\{z\in\R^d:\, |\Psi_{_X}(z)|>1\}$ does not intersect a certain ball about
	the origin of $\R^d$. Therefore, the inequality
	$\Re(1/\Psi_{_X}(z))\le |\Psi_{_X}(z)|^{-1}$, valid for all $z\in\R^d$, implies that
	\begin{align*}
		T_1 &\asymp\int_{|\Psi_{_X}(z)|\le1}\Re\left( \frac{1}{\Psi_{_X}(z)}\right)
			\frac{\d z}{\|z\|^{d-\alpha}},\\
			\intertext{and}
		T_2 &\asymp\int_{|\Psi_{_X}(z)|\ge1}\Re\left( \frac{1}{\Psi_{_X}(z)}\right)
			\frac{\e^{-\|z\|}}{\|z\|^{(d-\alpha+1)/2}}\,\d z
			\le\int_{|\Psi_{_X}(z)|\ge1} \frac{\e^{-\|z\|}}{%
			\|z\|^{(d-\alpha+1)/2}}\,\d z<\infty.
	\end{align*}
	This verifies that
	\[
		\int_{\R^d}\frac{\U_{_X}(\d x)}{1+|x|^\alpha}<\infty\qquad
		\Longleftrightarrow\qquad T_1<\infty,
	\]
	which completes the theorem in light of Theorem \ref{th:Range:Levy}
	and a real-variable argument that implies that $T_1<\infty$ iff
	$\int_{\|z\|\le1}\Re(1/\Psi_{_X}(z))\|z\|^{-d+\alpha}\,\d z<\infty$.
\end{proof}

\subsection{The graph of a L\'evy process}
Let $X:=\{X_t\}_{t\ge0}$ denote an arbitrary L\'evy process on $\R^d$,
not necessarily transient. It is easy to check that
\[
	Y_t := (t\,,X_t)\qquad[t\ge0]
\]
is a transient L\'evy process in $\R^{d+1}$. Moreover,
\[
	\cG_{_X} := \cR_{_Y}
\]
is the graph of the original L\'evy process $X$. The literature on
L\'evy processes contains several results  about the microscopic
structure of $\cG_{_X}$. Perhaps the most noteworthy
result of this type is the fact that
\begin{equation}\label{dim:graph:BM}
	\dim_{_{\rm H}}(\cG_{_X})=\nicefrac32
	\qquad\text{a.s.,}
\end{equation}
when $X$ denotes a
one-dimensional Brownian motion. In this section we compute
the large-scale Minkowski dimension of the same random set; in fact,
we plan to compute the maroscopic dimension of
the graph of a large class of L\'evy processes $X$.

The potential measure of the space-time process $Y$ is, in general,
\[
	\U_{_Y} (A\times B) := \E\left[\int_0^\infty\1_{A\times B}(s\,,X_s)\,\d s\right]
	=\int_A P_s(B)\,\d s,
\]
for all Borel sets $A\subseteq\R_+$ and $B\subseteq\R^d$, where
\[
	P_s(B) := \P\{ X_s\in B\}.
\]
Therefore, Theorem \ref{th:Range:Levy} implies that
\[
	\DimM\cG_{_X} = \inf\left\{ \alpha>0:\ \int_0^\infty \d s\int_{\R^d}
	\frac{P_s(\d x)}{1+s^\alpha + |x|^\alpha}<\infty\right\}
	\qquad\text{a.s.}
\]
In order to understand what this formula says, let us first prove the following
result.

\begin{lemma}\label{lem:0<Dim<1}
	If $X$ is an arbitrary L\'evy process on $\R^d$,
	then
	\[
		0\le \DimM(\cG_{_X})\le 1\qquad\text{a.s.}
	\]
\end{lemma}

\begin{proof}
	Since
	\[
		\int_0^1 \d s\int_{\R^d}
		\frac{P_s(\d x)}{1+s^\alpha + |x|^\alpha}
		\le \int_0^1 \d s\int_{\R^d} P_s(\d x)=1,
	\]
	it follows that
	\[
		\DimM\cG_{_X} = \inf\left\{ \alpha>0:\ \int_1^\infty \d s\int_{\R^d}
		\frac{P_s(\d x)}{s^\alpha + |x|^\alpha}<\infty\right\}
		\qquad\text{a.s.}
	\]
	The proposition follows because
	\[
		\int_1^\infty \d s\int_{\R^d}
		\frac{P_s(\d x)}{s^\alpha + |x|^\alpha}
		\le \int_1^\infty \frac{\d s}{s^\alpha}<\infty,
	\]
	whenever $\alpha>1$.
\end{proof}

It is possible to also show that, in a large number of cases,
the graph of a L\'evy process has macroscopic Minkowski dimension one, viz.,

\begin{proposition}\label{pr:Dim=1}
	Let $X$ be a L\'evy process on $\R^d$ such that
	$X_1\in L^1(\P)$ and $\E(X_1)=0$. Then,
	$\DimM(\cG_{_X})=1$ a.s.
\end{proposition}
Therefore, we can see from Proposition \ref{pr:Dim=1}
that the graph of one-dimensional
Brownian motion has macroscopic dimension 1,
yet it has microscopic Hausdorff dimension $\nicefrac32$; 
compare with \eqref{dim:graph:BM}.

\begin{proof}
	Lemma \ref{lem:0<Dim<1} implies that
	\begin{equation}\label{DimDim}
		\DimM(\cG_{_X}) = \inf\left\{ 0<\alpha<1:\ \int_1^\infty \d s\int_{\R^d}
		\frac{P_s(\d x)}{s^\alpha + |x|^\alpha}<\infty\right\}
		\qquad\text{a.s.,}
	\end{equation}
	where $\inf\varnothing:=1$. If $0<\alpha<1$, then
	\[
		\int_1^\infty \d s\int_{\R^d}
		\frac{P_s(\d x)}{s^\alpha + |x|^\alpha}
		\ge \int_1^\infty \d s\int_{|x|\le s}
		\frac{P_s(\d x)}{s^\alpha + |x|^\alpha} \ge 2^{-\alpha}
		\int_1^\infty\P\{|X_s|\le s\}\,\frac{\d s}{s^\alpha}.
	\]
	Because $\E(X_1)=0$, the law of large numbers for L\'evy processes (see, for example, Bertoin
	\cite[pp.\ 40--41]{Bertoin96} implies that
	$\P\{|X_s|\le s\}\to 1$ as
	$s\to\infty$. This shows that
	\[
		\int_1^\infty\P\{|X_s|\le s\}\,\frac{\d s}{s^\alpha}=\infty
		\quad\text{for every $\alpha\in(0\,,1)$,}
	\]
	and proves the lemma.
\end{proof}

Finally, let us prove that the preceding result is unimprovable in the following
sense: For every number $q\in[0\,,1]$, there exist a L\'evy process $X$ on $\R^d$
the macroscopic dimension of whose graph is $q$.

\begin{theorem}\label{th:Graph}
	If $X$ be a symmetric $\beta$-stable L\'evy process on $\R^d$
	for some $0<\beta\le 2$, then
	\[
		\DimM(\cG_{_X}) = \min\bigg\{ \frac{(2\beta-1)_+}{\beta}\,,1\bigg\}
		\qquad\text{a.s.}
	\]
\end{theorem}

Figure \ref{fig:f} below shows the graph of  the preceding function of $\beta$.

\begin{figure}[!ht]\centering
	\begin{tikzpicture}[scale=4]
		\draw[->,thin] (-0.1,0) -- (2.1,0) node[right] {$\beta$};
		\draw[->,thin] (0,-0.1) -- (0,1.1) node[above] {$g(\beta)$};
		\draw[ultra thin] (0,0) -- (0,0.05);\node at (-0.1,-0.1) {0};
		\draw[ultra thin] (0.5,-0.05) -- (0.5,0.05);\node at (0.5,-0.1) {$\nicefrac12$};
		\draw[ultra thin] (1,-0.05) -- (1,0.05);\node at (1,-0.1) {1};
		\draw[ultra thin] (2,-0.05) -- (2,0.05);\node at (2,-0.1) {2};
		\draw[ultra thin] (-0.05,1) -- (0.05,1);\node at (-0.1,1) {1};
		\draw[ultra thick,color=blue!80,rounded corners=1] (0,0) -- (0.5,0) --
			(0.55,  0.1818) -- (0.6,0.3333) -- (0.65, 0.4615)--
			(0.7,0.5714) -- (0.75,0.6667) -- (0.8,0.75) -- (0.85,0.8235)--
			(0.9,0.889)--(0.95, 0.9474) -- (1,1)-- (2,1);
		\draw[ultra thin,dashed] (0,1) -- (1,1);
		\draw[ultra thin,dashed] (1,0) -- (1,1);
		\draw[ultra thin,dashed] (2,0) -- (2,1);
	\end{tikzpicture}
	\caption{A plot of $g(\beta) := \min\big\{(2\beta-1)_+/\beta\,,1\big\}$}
	\label{fig:f}
\end{figure}
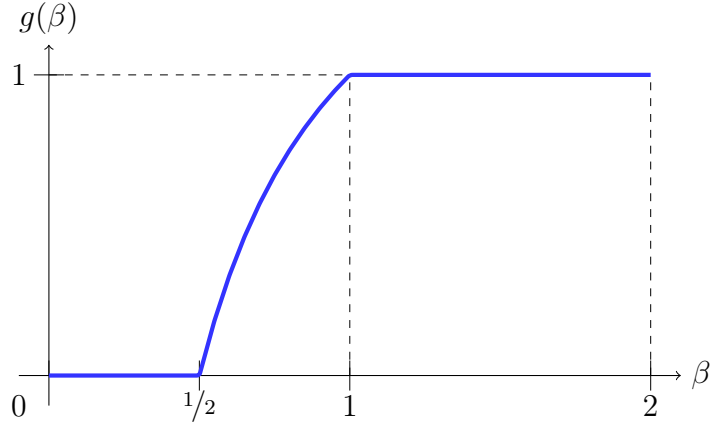
		
\begin{proof}[Proof of Theorem \ref{th:Graph}]
	If $\beta>1$, then $X_1$ is $\P$-integrable and $\E(X_1)=0$ by symmetry,
	and the result follows from Proposition \ref{pr:Dim=1}. In the remainder
	of the proof we assume that $0<\beta\le 1$.
	
	Let us observe the elementary estimate,
	\begin{equation}\label{elest}\begin{split}
		\int_1^\infty \d s\int_{\R^d}
			\frac{P_s(\d x)}{s^\alpha + |x|^\alpha}&\asymp
			\int_1^\infty \d s\int_{|x|<s}
			\frac{P_s(\d x)}{s^\alpha} +
			\int_1^\infty \d s\int_{|x|\ge s}
			\frac{P_s(\d x)}{|x|^\alpha}\\
		&=:\mathcal{T}_1 + \mathcal{T}_2.
	\end{split}\end{equation}
	
	For all $0<\alpha<1$,
	\[
		\mathcal{T}_1=\int_1^\infty\P\{ |X_s|<s\}\,\frac{\d s}{s^\alpha}
		=\int_1^\infty\P\{ |X_1|<s^{-(1-\beta)/\beta}\}\,\frac{\d s}{s^\alpha},
	\]
	by scaling. It is well known that $X_1$ has a bounded, continuous, and strictly positive
	density function on $\R^d$. 
	This fact implies that 
	\[
		\P\left\{|X_1|<s^{-(1-\beta)/\beta}\right\}\lesssim
		s^{-(1-\beta)/\beta}\qquad\text{uniformly for all $s>1$.}
	\]
	In particular, it follows that
	\begin{equation}\label{T1}
		\mathcal{T}_1<\infty
		\quad\text{iff}\quad1>\alpha>2-\beta^{-1}.
	\end{equation}
	
	Next, one might note that if $0<\alpha<1$, then
	\[
		\mathcal{T}_2=
		\int_1^\infty\E\left(|X_1|^{-\alpha};\, |X_1|\ge s^{1-(1/\beta)}\right)
		\frac{\d s}{s^{\alpha/\beta}},
	\]
	by scaling. Because $X_1$ has a strictly positive and bounded
	density in $B(0;2)$, the inequalities
	\[
		\E\left(|X_1|^{-\alpha};\, |X_1|\ge 1\right)\le
		\E\left(|X_1|^{-\alpha};\, |X_1|\ge s^{1-(1/\beta)}\right)\le
		\E\left(|X_1|^{-\alpha}\right)
	\]
	together imply that
	\begin{equation}\label{T2}
		\mathcal{T}_2<\infty
		\quad\text{iff}\quad0<\alpha<\beta.
	\end{equation}
	The theorem follows from \eqref{DimDim}, \eqref{elest},
	\eqref{T1}, and \eqref{T2}.
\end{proof}

\subsection{Application to subordinators}
Let us now consider the special case that the L\'evy process
$X$ is a subordinator. To be concrete, by the latter we mean that
$X$ is a L\'evy process on $\R$ such that $X_0=0$ and the sample function
$t\mapsto X_t$ is
a.s.\ nondecreasing. If we assume further that $\P\{X_1>0\}>0$, then it
follows readily that $\lim_{t\to\infty}X_t=\infty$ a.s.\ and hence
$X$ is transient. As is customary, one prefers to study subordinators
via their \emph{Laplace exponent} $\Phi_{_X}:\R_+\to\R_+$.
The Laplace exponent of $X$ is defined via the identity
\[
	\E\exp(-\lambda X_t)=\exp(-t\Phi_{_X}(\lambda)),
\]
valid for all $t,\lambda\ge0$.
It is easy to see that $\Phi_{_X}(\lambda)=\Psi_{_X}(i\lambda)$, where $\Psi_{_X}$ now
denotes [the analytic continuation, from $\R$ to $i\R$, of] the characteristic exponent of $X$.

\begin{theorem}\label{th:Range:Subord}
	If $\Phi_{_X}:\R_+\to\R_+$ denote the Laplace exponent of a subordinator
	$X$ on $\R_+$, then
	\[
		\DimM(\cR_{_X}) = \inf\left\{ 0<\alpha<1:
		\int_0^\infty\frac{\d y}{y^{1-\alpha}\,\Phi_{_X}(y)}<
		\infty \right\}\qquad\text{a.s.},
	\]
	where $\inf\varnothing:=1$.
\end{theorem}

Theorem \ref{th:Range:Subord} is the macroscopic analogue
of a theorem of Horowitz \cite{Horowitz68} (see also \cite{Bertoin:Subord} for more results)
 which gave a formula
for the microscopic Hausdorff dimension
of the range of a subordinator. The following
highlights a standard application of subordinators
to the study of level sets of Markov process; see Bertoin
\cite{Bertoin:Subord} for much more on this connection.

\begin{example}
	Let $X$ be a symmetric, $\beta$-stable process on $\R$
	where $1<\beta\le 2$. It is well known that
	$X^{-1}\{0\}:=\{s>0:\, X_s=0\}$ is a.s.\ nonempty, and
	coincides with the closure of the range of a stable subordinator
	$T:=\{T_t\}_{t\ge0}$ of index $1-\beta^{-1}$.  Since $T$ is c\`adl\`ag,
	$\cR_{_T}$ and its closure differ by at most a countable set a.s. Therefore,
	\[
		\DimM\left( X^{-1}\{0\}\right) = \inf\left\{0<\alpha<1:\
		\int_0^1 \frac{\d t}{y^{1-\alpha+1-(1/\beta)}}<\infty\right\}
		= 1 - \beta^{-1}\qquad\text{a.s.,}
	\]
	by Theorem \ref{th:Range:Subord}.
\end{example}

\begin{proof}[Proof of Theorem \ref{th:Range:Subord}]
	The proof uses as its basis an old idea which is basically a
	``change of variables for subordinators,'' and is loosely
	connected to Bochner's method of subordination (see Bochner \cite{Bochner55}).
	Before we get to that, let us first observe the following ready consequence
	of Theorem \ref{th:Range:Levy}:
	\[
		\DimM(\cR_{_X}) = \inf\left\{ 0<\alpha< 1:\
		\int_0^\infty x^{-\alpha}\,\U_{_X}(\d x) <\infty \right\}\qquad\text{a.s.}
	\]
	
	Now let us choose and fix some $\alpha\in(0\,,1)$, and
	let $Y :=\{Y_s\}_{s\ge0}$ be an independent $\alpha$-stable subordinator,
	normalized to satisfy $\Phi_{_Y}(x)=x^\alpha$
	for every $x\ge0$.
	Since $x^{-\alpha}=\int_0^\infty\exp(-sx^\alpha)\,\d s
	=\int_0^\infty\E\exp(-xY_s)\,\d s$, a few
	back-to-back appeals to the Tonelli theorem yield the following
	probabilistic change-of-variables formula:\footnote{
			The same argument shows that if $X$ and $Y$ are independent
			subordinators, then we have the change-of-variables formula,
			\[
				\int_0^\infty\frac{\U_{_X}(\d x)}{\Phi_{_Y}(x)}
				=\int_0^\infty\frac{\U_{_Y}(\d y)}{\Phi_{_X}(y)}.
			\]}
	\begin{align*}
		\int_0^\infty x^{-\alpha}\,\U_{_X}(\d x) &= \E\left[\int_0^\infty \U_{_X}(\d x)
			\int_0^\infty\d s\ \e^{-xY_s}\right]
			=\E\left[\int_0^\infty\d t\int_0^\infty\d s\ \e^{-X_tY_s}\right]\\
		&= \int_0^\infty\d t\int_0^\infty\d s\
			\E\left[\e^{-t\Phi_{_X}(Y_s)}\right]
			= \E\left[\int_0^\infty \frac{\d s}{\Phi_{_X}(Y_s)}\right] =
			\int_0^\infty\frac{\U_{_Y}(\d y)}{\Phi_{_X}(y)}.
	\end{align*}
	It is well-known that $\U_{_Y}(\d y)\ll\d y$ (or one can verify this directly using transition
	density or characteristic function of $Y$),
	and the Radon--Nikodym density $u_{_Y}(y) := \U_{_Y}(\d y)/\d y$---this
	is the socalled \emph{potential density of $Y$}---is given by
	$u_{_Y}(y) =c y^{-1+\alpha}$ for all $y>0$, where
	$c=c(\alpha)$ is a positive and finite constant [this follows from the scaling
	properties of $Y$]. Consequently, we see that
	$\int_0^\infty x^{-\alpha}\,\U_{_X}(\d x)<\infty$ for some $0<\alpha<1$
	if and only if $\int_0^\infty y^{-1+\alpha}\,\d y/\Phi_{_X}(y)<\infty$ for
	the same value of $\alpha$. The theorem follows from this.
\end{proof}

\section{Tall Peaks of Symmetric Stable Processes}

Let $B = \{B_t\}_{t \ge 0}$ be a standard Brownian motion.
For every $\alpha > 0$, let us consider the set
\begin{equation}\label{Eq:TB1}
	\cH_{_B}(\alpha) :=\left\{t\ge\e:\
	B_t\ge \alpha\sqrt{2t\log\log t} \right\}.
\end{equation}
In the terminology of Khoshnevisan, Kim, and Xiao \cite{KKX}, 
the random set $\cH_{_B}(\alpha)$ denotes the collection of \emph{the tall peaks of $B$ 
in length scale $\alpha$.}

Theorem \ref{Th:LIL} below follows from the law of the iterated  logarithm for 
Brownian motion for $\alpha \ne 1$. The critical case of $\alpha =1$ follows from Motoo
\cite[Example 2]{Motoo}.

\begin{theorem}\label{Th:LIL}
	$\mathcal{H}_{_B}(\alpha)$ is a.s.\ unbounded if
	$0<\alpha\le 1$ and is a.s.\ bounded if
	$\alpha>1$.
\end{theorem}

Recently, Khoshnevisan, Kim, and Xiao \cite{KKX} showed that the macroscopic 
Hausdorff dimension  of $\mathcal{H}_{_B}(\alpha)$ is 1 almost surely if $\alpha \le 1$.
Since the macroscopic 
Hausdorff dimension never exceeds the Minkowski dimension (see Barlow and Taylor 
\cite{BarlowTaylor}) Theorem \ref{Th:LIL} implies the following.

\begin{theorem}\label{Th;KKX}
	$\DimM(\cH_{_B}(\alpha)) = 1$ a.s.\
	for every $0<\alpha\le 1$.
\end{theorem}
Together, Theorems \ref{Th:LIL} and \ref{Th;KKX} imply that the tall peaks
of Brownian motion are macroscopic monofractals in the sense that
either $\DimM(\cH_{_B}(\alpha))=1$ or $\DimM(\cH_{_B}(\alpha))=0$.
In this section we extend the above results to facts about
all symmetric stable L\'evy processes. However, we are quick to point
out that the proofs, in the stable case,
are substantially more delicate than those in the Brownian case. 

Let $X = \{X_t\}_{t \ge 0}$ be a real-valued,
symmetric $\beta$-stable L\'evy process
for some $\beta\in(0\,,2)$. We have ruled out the case $\beta=2$
since $X$ is Brownian motion in that case, and there is nothing new
to be said about $X$ in that case. To be concrete, the process
$X$ will be scaled so that it satisfies
\begin{equation}\label{scale:stable}
	\E\exp(izX_t) = \exp(-t|z|^\beta)\qquad\text{for every $t\ge0$
	and $z\in\R$}.
\end{equation}

In analogy with \eqref{Eq:TB1}, for every $\alpha>0$,
let us consider the following set
\[
	\mathcal{H}_{_X}(\alpha) := \left\{ t\ge\e:\ X_t 
	\ge t^{1/\beta} (\Log t)^\alpha \right\}
\]
of tall peaks of $X$, parametrized by a ``scale factor'' $\alpha>0$. 
The following is a re-interpretation of a classical
result of Khintchine \cite{Khintchine}.

\begin{theorem}
	\label{th:Khintchine}
	$\mathcal{H}_{_X}(\alpha)$ is a.s.\ unbounded if $0<\alpha\le\beta^{-1}$,
	and it is a.s.\ bounded if $\alpha >\beta^{-1}$.
\end{theorem}

We include a proof for the sake of completeness.

\begin{proof}
	It suffices to study only the case that $\alpha  > \beta^{-1}$.
	The other case follows from the stronger assertion of
	Theorem \ref{th:Dim:stable:peaks}
	below.
	
	Recall from \cite[p.\ 221]{Bertoin96} that
	\begin{equation}\label{limlim}
		\varrho :=
		\lim_{\lambda\to\infty}\lambda^\beta\P\{ X_1> \lambda\}
	\end{equation}
	exists and is in $(0\,,\infty)$. Consequently,
	\begin{equation}\label{tail:stable}
		\P\{ X_t > t^{1/\beta}\lambda\} \asymp \lambda^{-\beta}\qquad
		\text{for all $\lambda\ge1$ and $t>0$}.
	\end{equation}
	
	Let
	\[
		X^*_t := \sup_{0\le s\le t}X_s
		\qquad\text{for all $t\ge0$.}
	\]
	The standard argument that yields the classical reflection principle also yields
	\[
		\P\left\{ X^*_t \ge \lambda\right\} \le 2\P\left\{
		X_t\ge\lambda\right\}
		\qquad\text{for all $t,\lambda>0$}.
	\]
	Therefore,  \eqref{tail:stable} implies that
	\[
		\P\left\{ X^*_t\ge \varepsilon t^{1/\beta}(\Log t)^\alpha\right\}
		\le 2 \P\left\{ X_t\ge \varepsilon
		t^{1/\beta}(\Log t)^\alpha\right\}\asymp (\Log t)^{-\alpha\beta},
	\]
	for all $t\ge\e$ and $\varepsilon>0$.
	This and the Borel--Cantelli lemma together show that, if
	$\alpha>\beta^{-1}$, then $X_t=o( t^{1/\beta}(\Log t)^\alpha)$
	as $t\to\infty$, a.s. In other words, $\mathcal{H}_{_X}(\alpha)$
	is a.s.\ bounded if $\alpha>\beta^{-1}$. This completes the proof.
\end{proof}

Theorem \ref{th:Khintchine} reduces the analysis of the peaks of
$X$ to the case where $\alpha\in(0\,,1/\beta]$. That case is
described by the following theorem, which is the promised extension
of Theorem \ref{Th;KKX} to the stable case.

\begin{theorem}\label{th:Dim:stable:peaks}
	If $0<\alpha\le\beta^{-1}$, then $\DimM(\cH_{_X}(\alpha))=1$ a.s.
\end{theorem}

\begin{proof}
	It suffices to prove that 
	\begin{equation}\label{goal:stable}
		\DimM(\cH_{_X}(\alpha))\ge 1\qquad\text{a.s.}
	\end{equation}
	Throughout the proof, we choose and fix a constant
	\[
		\gamma\in(0\,,1).
	\]
	
	Let us define an increasing sequence $T_1,T_2,\ldots,$ where
	\[
		T_j := 4^{\beta j^\gamma/\gamma} = 
		\exp\left( \frac{\beta \log (4) j^\gamma}{\gamma}\right),
	\]
	where ``$\log$'' denotes the natural logarithm.
	Let us also introduce a collection of intervals $\cI(1),\cI(2),\ldots,$ defined
	as follows:
	\[
		\cI(j) := \left[  T_j^{1/\beta}(\Log T_j)^\alpha
		~,~ 2T_j^{1/\beta}(\Log T_j)^\alpha\right).
	\]
	Finally, let us introduce events $\cE_1,\cE_2,\ldots$, where
	\[
		\cE_j := \left\{ \omega\in\Omega:\ 
		X_{T_j}(\omega) \in \cI(j)\right\}.
	\]
	
	According to \eqref{limlim},
	\begin{align*}
		\P(\cE_j) &= \P\left\{ X_{T_j}\ge T_j^{1/\beta}(\Log T_j)^\alpha\right\}
			-\P\left\{ X_{T_j} \ge 2 T_j^{1/\beta}(\Log T_j)^\alpha\right\}\\
		&\sim \frac{\varrho\left( 1 - 2^{-\beta}\right)}{(\Log T_j)^{\alpha\beta}}
			\qquad[\text{as }j\to\infty]\\
		&= \frac{\varrho\left( 1-2^{-\beta}\right)}{j^{\alpha\gamma\beta}}.
	\end{align*}
	For every integer $n\ge 1$, let us define
	\[
		W_n := \sum_{j=2^{n-1}}^{2^n-1}\1_{\cE_j}.
	\]
	It follows from the preceding that there exists an integer $n_0\ge1$ such that
	\begin{equation}\label{E(W_n)}
		\E(W_n) \gtrsim 2^{n(1-\alpha\beta\gamma)}\quad\text{uniformly
		for all $n\ge n_0$}.
	\end{equation}
	Next, we estimate $\E(W_n^2)$, which may be written in the
	following form:
	\begin{equation}\label{E(Wn2)}
		\E(W_n^2) = \E(W_n) + 2 \mathop{\sum\sum}\limits_{%
		2^{n-1}\le j < k < 2^n} \P(\cE_j\cap \cE_k).
	\end{equation}
	Henceforth, suppose $k>j$ are two integers between $2^{n-1}$ and $2^n-1$.
	
	Because $X$ has stationary independent increments,
	\begin{equation}\label{PPP}
		\P(\cE_j\cap\cE_k) \le \cP_j\times\cP_{j,k},
	\end{equation}
	where
	\begin{align*}
		\cP_j&=\P\left\{ X_{T_j} \ge T_j^{1/\beta}(\Log T_j)^\alpha\right\},\\
		\cP_{j,k}&=\P\left\{ X_{T_k-T_j} \ge T_k^{1/\beta}(\Log T_k)^\alpha
			- 2T_j^{1/\beta}(\Log T_j)^\alpha\right\}.
	\end{align*}
	In accord with \eqref{tail:stable},
	\begin{equation}\label{Pj}
		\cP_j =\P(\cE_j) \asymp j^{-\alpha\beta\gamma}.
	\end{equation}
	The analysis of $\cP_{j,k}$ is somewhat more complicated
	than that of $\cP_j$ and requires a little more work. 
	
	First, one might observe that
	\begin{equation}\label{Pjk}\begin{split}
		\cP_{j,k} &\le\P\left\{ X_{T_k-T_j} \ge (\Log T_k)^\alpha
			\left[T_k^{1/\beta}- 2T_j^{1/\beta}\right]\right\}\\
		&= \P\left\{ X_1 \ge k^{\alpha\gamma}
			\left[ \frac{T_k^{1/\beta} -2T_j^{1/\beta}}{(T_k-T_j)^{1/\beta}}\right]\right\}
			\quad[\text{by scaling}]\\
		&\le\P\left\{ X_1 \ge k^{\alpha\gamma}
			\left[ 1 - 2\left( \frac{T_j}{T_k}\right)^{1/\beta}\right]\right\}.
	\end{split}\end{equation}
	[The final inequality holds simply because $(T_k-T_j)^{1/\beta}
	\le T_k^{1/\beta}$.] 
	
	If $j$ and $k$ are integers in $[2^{n-1}\,,2^n)$ that satisfy $j\le k-k^{1-\gamma}$, then
	\[
		k^\gamma-j^\gamma = k^\gamma \left[ 1 - \left(\frac jk\right)^\gamma\right]
		\ge k^\gamma\left[ 1 - \left\{ 1 - k^{-\gamma} \right\}^\gamma\right]
		\ge \gamma.
	\]
	The preceding is justified by the following
	elementary inequality: $(1-x)^\gamma \ge 1-\gamma x$ for all $x\ge0$. 
	As a result, we are led to the following bound:
	\[
		1 - 2\left( \frac{T_j}{T_k}\right)^{1/\beta}
		= 1 -2\exp\left( -\frac{\log 4}{\gamma} \left[ k^\gamma-j^\gamma\right]\right)
		\ge \frac12,
	\]
	valid
	uniformly for all integers $j$ and $k$ that satisfy $k>j\ge k-k^{1-\gamma}$ and
	are between $2^{n-1}$ and $2^n-1$. Therefore, \eqref{tail:stable}  and \eqref{Pjk}
	together imply that
	\[
		\cP_{j,k} \le \P\left\{ X_1 \ge k^{\alpha\gamma}\right\}
		\lesssim k^{-\alpha\beta\gamma},
	\]
	uniformly for all integers $k>j$ that are in $[2^{n-1},\,2^n-1)$ and 
	satisfy $j\le k-k^{1-\gamma}$, and uniformly for every integer $n\ge n_0$.
	It follows from this bound, \eqref{PPP}, and \eqref{Pj} that
	\begin{equation}\label{stable1}
		\mathop{\sum\sum}\limits_{\substack{2^{n-1}\le j<k<2^n\\
		j\le k-k^{1-\gamma}}}\P(\cE_j\cap \cE_k) \lesssim
		\mathop{\sum\sum}\limits_{\substack{2^{n-1}\le j<k<2^n\\
		j\le k-k^{1-\gamma}}} (jk)^{-\alpha\beta\gamma}
		\lesssim 4^{n(1-\alpha\beta\gamma)},
	\end{equation}
	uniformly for all integers $n\ge n_0$. 
	
	On the other hand,
	\begin{align*}
		\mathop{\sum\sum}\limits_{\substack{2^{n-1}\le j<k<2^n\\
			j>k-k^{1-\gamma}}}\P(\cE_j\cap \cE_k) &
			\le\mathop{\sum\sum}\limits_{\substack{2^{n-1}\le j<k<2^n\\
			j>k-k^{1-\gamma}}}\P(\cE_k) \lesssim
			\mathop{\sum\sum}\limits_{\substack{2^{n-1}\le j<k<2^n\\
			j>k-k^{1-\gamma}}} k^{-\alpha\beta\gamma}\qquad\text{[by \eqref{Pj}]}\\
		&\lesssim\sum_{k=2^{n-1}}^{2^n-1} k^{1-\gamma-\alpha\beta\gamma}
			\lesssim 2^{n(2-\gamma-\alpha\beta\gamma)}\\
		&\le 4^{n(1-\alpha\beta\gamma)},
	\end{align*}
	since $\alpha\beta\le1$.
	Therefore, \eqref{stable1} implies that
	\[
		\mathop{\sum\sum}\limits_{2^{n-1}\le j<k<2^n}
		\P(\cE_j\cap \cE_k) \lesssim 4^{n(1-\alpha\beta\gamma)}.
	\]
	This and \eqref{E(Wn2)} together imply that
	\[
		\E(W_n^2) \le \E(W_n) + \left( \E[W_n]\right)^2,
	\]
	uniformly for all $n\ge n_0$. Because of \eqref{E(W_n)} and the condition
	$\alpha\beta\le1$,
	it follows that $\E(W_n)\gtrsim 1$, uniformly
	for all $n\ge 1$; in other words,
	$(\E[W_n])^2\gtrsim\E(W_n)$ for all $n\ge1$.
	Therefore, there exists a finite and positive constant $c$ such that
	\[
		\E(W_n^2) \le c\left( \E[W_n]\right)^2\qquad
		\text{for all $n\ge n_0$}.
	\]
	An appeal to the Paley--Zygmund inequality then yields
	the following: Uniformly for all integers $n\ge n_0$,
	\[
		\inf_{n\ge n_0}\P\left\{ W_n > \tfrac12\E(W_n) \right\} \ge (4c)^{-1}.
	\]
	From this and \eqref{E(W_n)} it immediately follows that
	\[
		\P\left\{ \limsup_{n\to\infty} n^{-1}
		\Log W_n \ge 1-\alpha\beta\gamma\right\} \ge (4c)^{-1}>0.
	\]
	The event in the preceding event is a tail event for the L\'evy process $X$.
	Therefore, the Hewitt--Savage 0--1 law implies that 
	\[
		\limsup_{n\to\infty} n^{-1}
		\Log W_n \ge 1-\alpha\beta\gamma 
		\qquad\text{a.s.}
	\]
	Because $\gamma\in(0\,,1)$ was arbitrary,  this proves that
	$\limsup_{n\to\infty} n^{-1}\Log W_n \ge 1$ a.s.,
	and \eqref{goal:stable} follows since $\DimM(\cH_{_X}(\alpha))
	\ge\limsup_{n\to\infty} n^{-1}\Log W_n$. This completes the proof
	of the theorem.
\end{proof}

\spacing{.8}\begin{small}

\bigskip
\noindent\textbf{Davar Khoshnevisan}.
Department of Mathematics, University of Utah,
Salt Lake City, UT 84112-0090,
\textcolor{purple}{\texttt{davar@math.utah.edu}}\\

\noindent\textbf{Yimin Xiao}.
\noindent Dept.\  Statistics \&\ Probability,
Michigan State University, East Lansing, MI 48824,
\textcolor{purple}{\texttt{xiao@stt.msu.edu}}
\end{small}

\end{document}